\documentclass[reqno]{amsart}
\usepackage{amsthm}
\usepackage{amsmath}
\usepackage{amssymb}
\usepackage{amsfonts}
\usepackage{amsthm}
\usepackage[all,cmtip]{xy}
\usepackage{txfonts}
\usepackage[margin=1.5in]{geometry}
\usepackage{verbatim}
\usepackage[pagebackref,pdftex,hyperindex]{hyperref}

\newtheorem{theorem}{Theorem}[section]
\newtheorem{definition}[theorem]{Definition}
\newtheorem{corollary}[theorem]{Corollary}
\newtheorem{proposition}[theorem]{Proposition}

\newtheorem{lemma}[theorem]{Lemma}
\newtheorem{remark}[theorem]{Remark}

\newtheorem{thm}{Theorem}

\newtheorem{cor}[thm]{Corollary}

\theoremstyle{remark}
\newtheorem*{ackn}{Acknowledgment} 

\numberwithin{equation}{section}       

\newcommand{\bK}{\bar{K}}

\newcommand{\hh}{\hat{h}}

\newcommand{\tf}{\tilde{f}}
\newcommand{\tT}{\tilde{T}}
\newcommand{\tX}{\tilde{X}}

\newcommand{\C}{\mathbb{C}}
\newcommand{\N}{\mathbb{N}}
\newcommand{\R}{\mathbb{R}}
\newcommand{\Z}{\mathbb{Z}}
\newcommand{\cC}{\mathcal{C}}
\newcommand{\cCper}{\mathcal{C}_{\mathrm{per}}}
\newcommand{\cCexc}{\mathcal{C}_{\mathrm{exc}}}
\newcommand{\cO}{\mathcal{O}}
\newcommand{\Dper}{D_{\mathrm{per}}}
\newcommand{\Dexc}{D_{\mathrm{exc}}}

\DeclareMathOperator{\Aut}{Aut}
\DeclareMathOperator{\Pic}{Pic}
\DeclareMathOperator{\NS}{NS}
\DeclareMathOperator{\DivGP}{Div}
\DeclareMathOperator{\Ind}{I}
\DeclareMathOperator{\Supp}{Supp}
\DeclareMathOperator{\Div}{div}
\DeclareMathOperator{\Dist}{dist}

\title{On the complex dynamics of birational surface maps defined over number fields}
\author{Mattias Jonsson \and Paul Reschke}
\address{Dept of Mathematics\\
  University of Michigan\\
  Ann Arbor, MI 48109--1043\\
  USA}
\email{mattiasj@umich.edu, preschke@umich.edu}
\date{}

\begin{document}

\setlength{\pdfpagewidth}{8.5in}
\setlength{\pdfpageheight}{11in}

\begin{abstract}
  We show that any birational selfmap of a complex projective surface 
  that has dynamical degree greater
  than one and is defined over a number field automatically satisfies
  the Bedford-Diller energy condition after a suitable birational
  conjugacy.
  As a consequence, the complex dynamics of the map is well-behaved.
  We also show that there is a well-defined canonical height function.
\end{abstract}

\maketitle

\section{Introduction}
Much work in higher-dimensional complex dynamics has centered around
the problem of finding interesting invariant measures. In this paper,
we consider a birational surface map $f$, that we furthermore assume
has dynamical degree $\lambda>1$. 
This setting includes polynomial automorphisms $f\in\Aut(\C^2)$, which
were studied in detail by Bedford-Lyubich-Smillie~\cite{BS1,BS2,BS3,BLS1,BLS2}. 
In this case, there is a unique invariant probability measure $\mu$ of maximal entropy $\log\lambda$. 
It can be constructed as the product $\mu= T^+\wedge T^-$ of positive closed $(1,1)$-currents 
$T^{\pm}$ satisfying $f^*T^+=\lambda T^+$ and 
$f_*T_-=\lambda T^-$, and has many nice properties.
Cantat~\cite{Can-B} later proved analogous results in the case of automorphisms of compact complex surfaces.

A general birational selfmap $f\colon X\dashrightarrow X$ of a compact
complex surface $X$ typically has a nonempty \emph{indeterminacy set}
$\Ind_f$, a finite set where $f$ fails to be defined. Forn{\ae}ss and
Sibony (see~\cite{Sib99}) introduced the important notion of \emph{algebraic
stability}, meaning that the forward orbits of the indeterminacy points
for $f^{-1}$ are well-defined, that is, they don't intersect the
indeterminacy set of~$f$. 

Diller and Favre~\cite{DF} showed that $f$ becomes algebraically stable after a finite number of carefully chosen point blowups. 
They also proved that when $f$ is algebraically stable, it admits positive closed currents $T^\pm$ satisfying the invariance relations above. These currents are unique up to scaling, and it is tempting to define a probability measure $\mu:=T^+\wedge T^-$. However, Buff~\cite{Buff} showed that $T^\pm$ may be supported on a pluripolar set, in which case it is not clear that 
their product is well-defined.

In~\cite{BD-A}, Bedford and Diller introduced a quantitative version of algebraic stability,
\begin{equation*}
  \sum_{k \geq 0} \lambda^{-k} \log \Dist \big(f^k(x),\Ind_f \big) > -\infty
  \quad\text{for every $x\in\Ind_{f^{-1}}$},\tag{BD}
\end{equation*}
under which the product $\mu=T^+\wedge T^-$ is well-behaved and has good properties, as
we shall specify later.
However, examples by Favre~\cite{Fav-A} show that~(BD) may fail.
Further, except for special classes such as polynomial automorphisms, 
it may be difficult to verify whether or not~(BD) holds.

\smallskip
Our main result gives a new large class of examples where the complex dynamics is well
behaved.
\begin{thm}\label{MainThm}
  Let $f:X\dashrightarrow X$ be a birational selfmap of a smooth complex projective surface,
  with dynamical degree $\lambda>1$. 
  Assume that $X$ and $f$ are defined over a number field. 
  Then, up to birational conjugacy, \(f\) satisfies the Bedford-Diller energy condition~(BD).
\end{thm}
As a consequence, we get:
\begin{cor}\label{MainCorCurr}
  Let $f:X\dashrightarrow X$ be an algebraically stable birational selfmap of a 
  smooth complex projective surface, with dynamical degree $\lambda>1$. 
  Assume that $X$ and $f$ are defined over a number field.
  Then the currents \(T^+\) and \(T^-\) are laminar, do not charge any pluripolar set,
    and admit a geometric intersection~$\mu$ satisfying \(f_*\mu=\mu\).
\end{cor}
The measure $\mu$ is invariant under birational conjugation and therefore
defined even without assuming algebraic stability.
As the next result shows, it has many nice properties.
\begin{cor}\label{MainCorMeas}
  Let $f:X\dashrightarrow X$ be a birational selfmap of a 
  smooth complex projective surface, with dynamical degree $\lambda>1$. 
  Assume that $X$ and $f$ are defined over a number field.
  Then there exists a unique probability measure \(\mu\) on \(X\) with the following properties:
  \begin{itemize}
  \item[(1)]
    \(\mu\) does not charge any pluripolar set on $X$; in particular, it does not 
    charge any curve or point on \(X\);
  \item[(2)]
    $f_*\mu = \mu$, and $\mu$ is mixing and hyperbolic for \(f\); its Lyapunov
    exponents $\chi^\pm=\chi^\pm(\mu,f)$ satisfy
    \begin{equation*}
      \chi^-<-\frac12\log\lambda
      \quad\text{and}\quad
      \chi^+>\frac12\log\lambda;
    \end{equation*}  
  \item[(3)]
    $h_{\rm{top}}(f) = h_\mu(f) = \log(\lambda)$, so that 
    $\mu$ is a measure of maximal entropy for $f$;
  \item[(4)] 
    the saddle periodic points for $f$ are equidistributed for $\mu$ in the sense that
    \[\lambda^{-n} \sum_{p \in \mathrm{SPer}_n} \delta_p \xrightarrow[n \rightarrow \infty]{} \mu.\]
  \end{itemize}
  Furthermore, (1) and (4) imply that the set of periodic points for \(f\) is Zariski dense in \(X\).
\end{cor}

The properties in Corollaries~\ref{MainCorCurr} and~\ref{MainCorMeas} are consequences of the energy condition~(BD), as proved
by Bedford-Diller~\cite{BD-A}, Dujardin~\cite{Duj06} 
and Diller-Dujardin-Guedj~\cite{DDG2}.
Indeed, the currents $T^\pm$ are essentially birational invariants, so
it suffices to work on a birational model where the energy
condition holds. See~\S\ref{S204} for details.

We remark that Junyi Xie has suggested an application of Corollary~\ref{MainCorMeas} to the result (in~\cite{XiePerpts}) that a birational selfmap $f$ on a projective surface defined over any algebraically closed field of characteristic zero must have a Zariski dense set of periodic points if $\lambda(f)>1$: one can combine the final conclusion of Corollary~\ref{MainCorMeas} with a specialization argument for a family of models of $f$ to give a proof which bypasses the use of a theorem due to Hrushovski in the argument in~\cite{XiePerpts}.

Let us emphasize that in condition~(BD) we measure distances to the indeterminacy set.
Estimating the corresponding sum when $\Ind_f$ is replaced by a general finite subset 
defined over $\bar{\mathbb{Q}}$ is probably very difficult: see~\cite{Silverman93}
for related results in one variable.

\smallskip
While we limit ourselves to birational surface maps in this paper, we expect that 
there are other instances where being defined over a number field has good consequences
for the complex dynamics. 
As a simple one-dimensional example, note that Roth's theorem implies that 
a rational selfmap of the Riemann sphere that is defined over a number field
has no Cremer points: see~\cite[\S11]{Milnor}. With regard to potential applications of the methods in this paper to maps in higher dimensions, we remark that De Thelin and Vigny~\cite{dTV10} have extended~(BD) to certain birational maps on $\mathbb{P}^k$ ($k\ge2$) and shown that the condition guarantees nice dynamical properties for these maps as well.

\smallskip
In the setting of Theorem~\ref{MainThm}, let $K$ be a number field over which $f$ and $X$
are defined. In Theorem~\ref{T101} we prove a result that is stronger than~(BD) in two ways. 
First, the summability condition in~(BD) holds
for every rational point $x\in X(K)$ with well-defined forward orbit. 
Second, we can replace the distance with the $v$-adic distance, for every place
$v$ of $K$.
However, we are certainly not aware of any analogue of Corollaries~\ref{MainCorCurr} and~\ref{MainCorMeas} 
for non-Archimedean fields.

\medskip
The main idea in the proof of Theorem~\ref{MainThm} is to exploit local and global heights. 
As a byproduct of our analysis, we find that birational surface maps admit canonical heights.

\begin{thm}\label{HeightThm}
  Let $X$ be a smooth projective surface and $f:X\dashrightarrow X$ a birational selfmap,
  both defined over a number field $K$. Assume that $f$ has 
  dynamical degree $\lambda>1$. Then, up to birational conjugation, the limit 
  \begin{equation*}
    \hh^+(x) := \lim_{n \rightarrow \infty} \lambda^{-n} h_{\theta^+}(f^n(x))
  \end{equation*}
  exists and is finite and non-negative for all points \(x \in X(\bar K)\) with well-defined forward orbit.
  Further, $\hh(x)=0$ unless the Zariski closure of the orbit of $x$ intersects $\Ind_f$.
\end{thm}
Here $\theta^+\in\Pic(X)_\R$ is the unique (up to scaling) nonzero nef
class satisfying $f^*\theta^+=\lambda\theta^+$ and $h_{\theta^+}$ is the corresponding
Weil height.
We have $\hh^+\circ f=\lambda\hh^+$, so it makes sense
to call $\hh^+$ a canonical height for $f$. There is also a 
canonical height $\hh^-$ for $f^{-1}$. These canonical heights generalize  
constructions in~\cite{Silverman91,Silverman94,Kaw06,Kaw08}; see also~\cite{Kaw13,JW}.

Unfortunately, we are unable to rule out the possibility that $\hh^\pm$ are identically
zero. However, if, for example,  
$\hh^+\equiv0$, then we show that $\lambda^{-n}h_A\circ f^n\to0$ on
$X(\bar K)$ for every ample line bundle $A$. 
Since Xie~\cite{XiePerpts} proved that $f$ admits Zariski dense orbits, 
this would contradict a recent conjecture by Silverman~\cite{Sil14}.

\medskip
Let us be somewhat precise about how we use heights to prove the generalized
Bedford-Diller condition~($\mathrm{BD}_v$), with respect to a place $v$ of $K$.
Theorem~\ref{HeightThm} shows that the global height does
not grow faster than $\lambda^n$, but this is not enough for~($\mathrm{BD}_v$),
and also does not seem to be enough for the weaker energy 
conditions considered by Diller, Dujardin and Guedj~\cite{DDG2}.
Instead we proceed as follows.

For simplicity, assume that the invariant class $\theta^+$ is ample and 
fix a point $x\in X(K)$ with well-defined orbit. 
Height considerations imply $\sum_{n=1}^\infty \lambda^{-n} \alpha_n>-\infty$,
where $\alpha_n=h_{\theta^+}(f^n(x))-h_{f^*\theta^+}(f^{n-1}(x))$.
A well chosen decomposition of the heights involved into local heights
gives rise to a decomposition of $\sum_n\lambda^{-n}\alpha_n$ into a sum 
over all the places $v$ of $K$. We show that each term in the decomposition is a finite lower bound for the sum defining the condition~($\mathrm{BD}_v$). 

In general, $\theta^+$ is not ample. To make the argument work, we need to 
understand the dynamics of the \emph{base curves} of $\theta^+$, that is, 
the irreducible curves $C$ orthogonal to $\theta^+$. This is done in
Theorem~\ref{basecurvesmain}.

\medskip
The paper is organized as follows. 
In~\S\ref{S201} we recall facts about the dynamics of birational surface maps, 
heights, and the Bedford-Diller energy condition. 
For the convenience of the reader, we treat the technically simpler case when 
$f:\mathbb{P}^2\dashrightarrow\mathbb{P}^2$ is algebraically stable in~\S\ref{S301}.
In the general case, the invariant classes $\theta^\pm$ are studied in detail in~\S\ref{S202}.
In~\S\ref{S203} we establish the generalized $v$-adic energy condition, as well as 
Theorem~\ref{HeightThm}, while the consequences for 
complex dynamics (Corollaries~B and~C) are treated in~\S\ref{S204}.

\begin{ackn} The authors would like to thank Jeff Diller, Romain Dujardin, Joe Silverman, and the referees for useful comments.
The first author was supported by NSF grant DMS-1266207. The second author was supported by NSF grants DMS-0943832 and DMS-1045119.
\end{ackn}

\section{Background}\label{S201}
Fix a rational surface \(X\) and a birational self-map \(f:X \dashrightarrow X\). In~\S\S\ref{BGactions}--\ref{BGstability}, we can take \(X\) and \(f\) to be defined over an arbitrary algebraically closed field. In~\S\ref{BGenergy}, we take \(X\) and \(f\) to be defined over \(\C\), whereas in~\S\S\ref{BGheights}--\ref{S102} we work over a number field.

When \(D\) is an divisor (resp., $\R$-divisor) on \(X\), we let \([D]\) denote the class of \(D\) in \(\Pic(X) \cong \NS(X)\) (resp., $\Pic(X)_{\R}\simeq\NS(X)_{\R}$). For classes \(\alpha\) and \(\beta\) in \(\Pic(X)_\R\), \((\alpha \cdot \beta)\) denotes the intersection pairing; we also write $(D\cdot\beta):=([D]\cdot\beta)$ 
and $(C\cdot D):=([C]\cdot[D])$ for divisors $C$ and $D$. 
The notation \(\alpha_1 \leq \alpha_2\) for $\alpha_1,\alpha_2\in\Pic(X)_{\R}$ means that \(\alpha_2-\alpha_1\) is pseudoeffective; hence
$(\alpha_1\cdot \beta) \leq (\alpha_2\cdot \beta)$
if \(\alpha_1 \leq\alpha_2\) and \(\beta\in \Pic(X)_\R\) is nef.

\subsection{Pullbacks of line bundles}\label{BGactions}
The following definitions and facts are presented in detail in~\cite{DF} 
and~\cite{DJS}\footnote{While~\textit{loc.\ cit.} only treat complex surfaces, the relevant arguments are valid over any algebraically closed field.}; the statements for \(f\) apply also to every (forward or backward) iterate \(f^n\) (\(n \in \mathbb{Z}\)). 

We let \(\Ind_f\) denote the \textit{indeterminacy set} of \(f\)---i.e., the finite set of points on \(X\) to which \(f\) does not extend as a morphism.
An irreducible curve \(C\) on \(X\) is \textit{\(f\)-exceptional} if \(f(C \smallsetminus \Ind_f)\) is a point; so \(f\) defines an isomorphism on the complement of the union of \(\Ind_f\) with all of the finitely many \(f\)-exceptional curves. 

There is a smooth surface \(Y\) with birational morphisms (i.e., finite compositions of point blow-ups) \(\pi,g:Y \rightarrow X\) such that the following diagram commutes,
\[\xymatrix{
Y \ar[d]_\pi \ar[rd]^g\\
X \ar@{-->}[r]^f &X}\]
i.e., \(f = g \circ \pi^{-1}\). 
The definitions below will not depend on the choice of $(Y,\pi)$, 
but we note that $Y$ can be chosen minimally by successively blowing up the 
indeterminacy points of $f$. In particular
we may assume $\pi$ (resp., $g$) is an isomorphism above $X\smallsetminus\Ind_f$
(resp., above $X\smallsetminus\Ind_{f^{-1}}$).

The pullback and pushforward of a divisor $D$ under $f$ are defined as 
\(f^*D := \pi_*g^*D\) and  \(f_*D:= g_*\pi^*D\), respectively.
Note that $f_*=(f^{-1})^*$.  
These operations preserve effective and nef divisors, and descend to linear maps 
on $\Pic(X)$ that are adjoint under the intersection form in the sense that
\[(f^*\alpha \cdot \beta) = (\alpha \cdot f_*\beta)\]
for any \(\alpha\) and \(\beta\) in \(\Pic(X)_\R\). 
Moreover, it follows from Theorem 3.3 in~\cite{DF} (the ``push-pull formula'') that \(f^*\) is intersection-increasing in the sense that
\[(f^*\alpha \cdot f^*\alpha) \geq (\alpha \cdot \alpha)\]
for any \(\alpha \in \Pic(X)_\R\). 
In particular, if $\alpha$ is big and nef, then so are $f^*\alpha$ and $f_*\alpha$.
By the ``negativity lemma''  (see Theorem 3.3 in~\cite{DF}) we also have
\begin{equation}\label{e102}
  g^*\alpha \leq \pi^*f^*\alpha=\pi^*\pi_*g^*\alpha
\end{equation}
for $\alpha\in\Pic(X)_{\R}$ nef; indeed, 
$\pi^*f^*\alpha-g^*\alpha$ is represented by an effective $\pi$-exceptional
$\R$-divisor.

We let \(R_f\) denote the \textit{ramification divisor} for \(f\)---i.e., the effective divisor on \(X\) characterized by satisfying \([R_f] = K_X-f^*K_X\) and having as its support the union of all of the \(f\)-exceptional curves. Note that \(R_f=\pi_*R_g\), where \(R_g\) is the ramification divisor for \(g\). 

When \(D\) is a prime divisor, the support of \(f^*D\) is in general a union of \(f\)-exceptional curves and possibly a single prime divisor \(D'\) such that the Zariski closure of \(f(D' \smallsetminus \Ind_f)\) is \(D\). 
We define the \textit{strict transform} of \(D\) under \(f\) to be \(f^\#D :=D'\) if \(D'\) exists and \(f^\#D :=0\) otherwise. 

For \(x \in \Ind_f\), we define \(f(x) := g(\pi^{-1}(\{x\}))\)---i.e., the Zariski closure of \((f^{-1})^{-1}(\{x\})\), or, equivalently, the union of all of the \(f^{-1}\)-exceptional curves which map to \(x\) under \(f^{-1}\). Note that \(f(x)\) is a connected and non-empty curve.
We will view it as a nonzero reduced divisor on $X$.

\subsection{Pullbacks of sections}\label{BGsections}
As recalled above, we can pull back line bundles by rational maps.
We will also need to pull back global sections; the next two results 
tell us how to do so.
\begin{proposition}\label{P101}
  Let $L\in\Pic(X)$ and let $\sigma\in H^0(X,L)$ be a nonzero global section.
  Then there exists a section $\sigma'\in H^0(X,f^*L)$, unique up to a 
  multiplicative constant, such that  
  $\Div(\sigma')=f^*\Div(\sigma)$. Furthermore: 
  \begin{itemize}
  \item[(1)]
    if $x\in X\smallsetminus \Ind_f$, then $\sigma'(x)=0$ iff $\sigma(f(x))=0$;
  \item[(2)]
    if $x\in \Ind_f$, $(L\cdot f(x))>0$, and $\Div(\sigma)$ is not supported on any \(f^{-1}\)-exceptional curve, then $\sigma'(x)=0$.
  \end{itemize}
\end{proposition}
\begin{proof}
  Write $L':=f^*L$. 
  Uniqueness is clear: if the divisor of $\sigma'_i\in H^0(X,L')$ is equal to 
  $f^*\Div(\sigma)$ for $i=1,2$, then $\sigma'_1/\sigma'_2$ is a rational function 
  on $X$ without zeros or poles, and hence constant.

  The prove existence, set $D:=\Div(\sigma)$ and note that $D':=f^*D$ is an 
  effective divisor representing $L'$. We can find a section $\sigma'$ of $L'$
  on $X\smallsetminus\Ind_f$ such that $\Div(\sigma')=D'$.
  Since $\Ind_f$ has codimension two, $\sigma'$ extends as a global section of $L'$
  still satisfying $\Div(\sigma')=D'$.

  Finally,~(1) follows immediately from 
  $\Div(\sigma')=f^*\Div(\sigma)$, and~(2) holds since the assumptions imply that $x$ is in the Zariski closure of $f^{-1}(\Supp(D))$.
\end{proof}
\begin{lemma}\label{L103}
  If $\sigma,\tau\in H^0(X,L)$ are nonzero sections, 
  with associated sections $\sigma',\tau'\in H^0(X,f^*L)$ given by 
  Proposition~\ref{P101}, then the function
  \begin{equation*}
    x\mapsto\frac{\tau(f(x))/\sigma(f(x))}{\tau'(x)/\sigma'(x)}
  \end{equation*}
  on $X\setminus(\Ind_f\cup\mathrm{Supp}(\Div(\sigma'\tau')))$ is a nonzero constant.
\end{lemma}
\begin{proof}
  Note that $\tilde\sigma:=\pi^*\sigma'/g^*\sigma$ 
  and $\tilde\tau:=\pi^*\tau'/g^*\tau$ both 
  define global sections of the line bundle $M:=\pi^*f^*L-g^*L$ on
  $Y$ whose associated divisor has $\pi$-exceptional support.
  Thus $\tilde\tau/\tilde\sigma$ defines a rational function on $X$
  (and on $Y$) that is regular and nonvanishing outside $\Ind_f$. Since the latter set
  is of codimension two, $\tilde\tau/\tilde\sigma$ must be constant,
  which concludes the proof.
\end{proof}

\subsection{Algebraic stability}\label{BGstability} By Theorem 0.1 in~\cite{DF}, we can successively blow up points on \(X\) to conjugate \(f\) to a map, called \textit{algebraically stable}, satisfying
\vspace{6pt}
\\
\indent \indent (AS1) \indent \indent \(f^k(\Ind_{f^{-1}}) \cap \Ind_f = \emptyset\) for every \(k \geq 0\).
\vspace{6pt}
\\
Condition (AS1) is equivalent to the property that \((f^n)^* = (f^*)^n\) for every \(n \in \mathbb{Z}\). We assume henceforth that \(f\) satisfies (AS1). The spectral radius of $f^*$ acting on $\Pic(X)_\R$ is then independent of $X$
and equal to\footnote{The dynamical degree can be defined and studied 
  independently of the choice of an algebraically stable model: see e.g.~\cite{RuSh,Friedland,deggrowth}.}
 the (first) \textit{dynamical degree} $\lambda=\lambda(f)$ of \(f\).
We have $\lambda(f^n)=\lambda^{|n|}$ for every $n\in\mathbb{Z}$.

We assume henceforth that \(\lambda>1\). Then Theorem 0.3 in~\cite{DF} shows that 
there are nef classes \(\theta^+=\theta^+(f)\) and \(\theta^-=\theta^-(f)\) in \(\Pic(X)_\R\), unique up to scaling, such that \(f^*\theta^+ = \lambda \theta^+\), \(f_*\theta^- = \lambda \theta^-\), and \((\theta^+ \cdot \theta^-) >0\). Furthermore, by Theorem 0.5 in~\cite{DF}, for any $\alpha\in\Pic(X)_\R$ we have
\begin{equation}\label{e101}
  \lim_{n\to\infty}\lambda^{-n}f^{n*}\alpha
  =\frac{(\alpha\cdot\theta^-)}{(\theta^+\cdot\theta^-)}\theta^+.
\end{equation}
Note that (AS1) is satisfied by every iterate \(f^n\) (\(n \in \mathbb{Z}\)). Up to scaling, we have  $\theta^{\pm}(f^{-1})=\theta^{\mp}(f)$ and \(\theta^{\pm}(f^n)=\theta^{\pm}(f)\) for $n \in \N$. 

By Proposition 4.1 in~\cite{BD-A}, we can successively blow down curves on \(X\) to conjugate \(f\) to a map which maintains (AS1) and also satisfies
\vspace{6pt}
\\
\indent \indent (AS2) \indent \indent \((\theta^+ \cdot f(x)) > 0\) for every \(x \in \Ind_f\) and \((\theta^- \cdot f^{-1}(x)) > 0\) for every \(x \in \Ind_{f^{-1}}\).
\vspace{6pt}
\\
We assume henceforth that \(f\) satisfies both~(AS1) and~(AS2).

\begin{proposition}
Every iterate \(f^n\) (\(n \in \mathbb{Z}\)) also satisfies (AS2).
\end{proposition}

\begin{proof}
By symmetry it suffices to consider the case $n>0$, which we
treat by induction on $n$. 
Thus suppose $n>1$ and \(x \in \Ind_{f^n}\). 
If $f(x)\cap\Ind_{f^{n-1}}=\emptyset$, then $x\in\Ind_f$
and $f^n(x)=f^{n-1}_*f(x)$, so
\[
(\theta^+\cdot f^n(x))
=(\theta^+\cdot f^{n-1}_*f(x)) 
=(f^{(n-1)*}\theta^+\cdot f(x)) 
=\lambda^{n-1}(\theta^+\cdot f(x))>0.
\]
If instead $f(x)\cap\Ind_{f^{n-1}}\ne\emptyset$, 
then $f^n(x)\ge f^{n-1}(y)$ for any 
$y\in f(x)\cap\Ind_{f^{n-1}}$,
and hence
\[(
\theta^+\cdot f^n(x))
\ge(\theta^+\cdot f^{n-1}(y))>0.
\]
Thus $(\theta^+\cdot f^n(x))>0$ for $n\ge 1$ and $x\in\Ind_{f^n}$. 
A similar argument shows that $(\theta^-\cdot f^{-n}(x))>0$
for $n\ge 1$ and $x\in\Ind_{f^{-n}}$. This completes the proof.
\end{proof}

In the following, we may replace \(f\) by a forward iterate to simplify certain statements and arguments.

\begin{definition}
An irreducible curve \(C\) on \(X\) is a \textbf{base curve} for \(f\) (resp., \(f^{-1}\)) if \((C \cdot \theta^+)=0\) (resp., \((C \cdot \theta^-)=0\)). We let \(\cC^+\) (resp., \(\cC^-\)) denote the set of base curves for \(f\) (resp. \(f^{-1}\)).
\end{definition}

We assume henceforth that \(f\) is \textit{strictly birational}---i.e., not birationally conjugate to an automorphism. By Theorem 0.4 in~\cite{DF}, this assumption is equivalent to $\theta^+$ being big, that is, \((\theta^+ \cdot \theta^+) > 0\). It follows from the Hodge index theorem that each curve \(C \in \cC^+\) must have negative self-intersection; so \(C\) is the only effective divisor representing \([C] \in \Pic(X)\).  The following proposition only uses the fact that \(\theta^+\) is big and nef.
\begin{proposition}[\cite{Kaw08}, Proposition 1.3]\label{Kawdecomp}
  There are finitely many curves in \(\cC^+\). Further, 
  \begin{equation*}
    \theta^+=A+[D],
  \end{equation*}
  where $A\in\Pic(X)_{\R}$ is ample and $D$ is an effective $\R$-divisor on $X$ whose
  support is equal to the union of all the curves in $\cC^+$.
\end{proposition}

The decomposition in Proposition~\ref{Kawdecomp} is neither unique nor canonical. However, every decomposition \(\theta^+ = A + [D]\) with \(A \in \Pic(X)_\R\) ample and $D$ an effective \(\R\)-divisor must have the property that the support of \(D\) contains every \(C \in \cC^+\). On the other hand, \(\theta^+\) may be representable by an effective \(\R\)-divisor whose support does not contain every curve in \(\cC^+\).

\subsection{Complex dynamics and finite energy}\label{BGenergy}
We now take \(X\) and \(f\) to be defined over \(\C\). 
Assume $f$ satisfies~(AS1) but not necessarily~(AS2).
By Theorem 0.5 in~\cite{DF} and Theorem 2.6 in~\cite{BD-A}, the cohomology classes \(\theta^+\) and \(\theta^-\) contain unique positive currents \(T^+\) and \(T^-\) satisfying \(f^*T^+=\lambda T^+\) and \(f_*T^-=\lambda T^-\). 
These currents have zero Lelong number outside
$\bigcup_{n\ge1}\Ind_{f^{\pm n}}$.
If \(\omega^{\pm}\) is a smooth form in the cohomology class \(\theta^{\pm}\),
then \(T^{\pm} = \omega^{\pm} + dd^cG^{\pm}\) with \(G^{\pm}\) quasi-psh.

In this setting, Bedford and Diller~\cite{BD-A} (see also~\cite{DillerP2}) introduced the condition
\begin{equation*}
  \sum_{k \geq 0} \lambda^{-k} \log \Dist \big(f^k(\Ind_{f^{-1}}),\Ind_f \big) > -\infty,\tag{BD}
\end{equation*}
which is stronger than (AS1). 
(Here \(\Dist\) is any distance function on \(X\) induced by a Hermitian metric.)  
Assuming that~(AS2) holds, Theorem 4.3 in~\cite{BD-A} shows that (BD) is equivalent to the condition \(G^+(x) > -\infty\) for every \(x \in \Ind_{f^{-1}}\).

When \(f\) satisfies~(AS2) and~(BD), the complex dynamics of $f$ is well-behaved. 
Indeed, the main result in~\cite{BD-A} states that \(\mu = T^+ \wedge T^-\) is a well-defined probability measure that charges no algebraic set and is invariant, mixing, and hyperbolic for \(f\). 
Dujardin~\cite{Duj04,Duj06} improved upon this result, by showing that $T^\pm$ are laminar
currents; 
that the intersection $\mu=T^+\wedge T^-$ is geometric and that $\mu$ has a local product structure
in the sense of Pesin.
Further, $\mu$ has maximal entropy $\log\lambda$, describes the distribution of saddle
periodic points, and has Lyapunov exponents satisfying $|\chi|\ge\frac12\log\lambda$.
Finally, Diller, Dujardin and Guedj, in their work~\cite{DDG2} 
on rational maps of small topological degree, proved that $\mu$ does not charge pluripolar sets.
They also extended this result to allow for slightly weaker conditions than~(BD). However, examples by Favre~\cite{Fav-A} and Bedford~\cite{Bed03} (ex. 5) show that (BD) (and even the weaker conditions in \cite{DDG2}) can fail to hold, and examples by Buff \cite{Buff} show that \(T^+\) and \(T^-\) can charge pluripolar sets.

In~\S\ref{PfMainThm}, we will prove Theorem~\ref{MainThm} by showing that \(f\) satisfies (BD) when \(X\) and \(f\) are defined over a number field (with a given embedding into \(\C\)) and $f$ satisfies~(AS2).

The currents $T^{\pm}$ are essentially birationally invariant. 
More precisely, suppose we have a birational morphism $\rho\colon\tX\to X$
and that the conjugate birational maps $f:X\dashrightarrow X$ and 
$\tf:\tX\dashrightarrow\tX$ both satisfy~(AS1). Let $T^\pm$ and
$\tT^\pm$ be the invariant currents. 
Then Proposition~2.8 in~\cite{BD-A} shows that, up to scaling, we have
$\rho_*\tT^\pm=T^\pm$. Since $\tT^\pm$ has zero Lelong number outside
the countable set $\bigcup_{n\ge1}\Ind_{f^n}$, it follows that
$\tT^\pm$ is the strict transform of $T^\pm$.

\subsection{Height functions}\label{BGheights}
In the next two sections, \(X\) is any irreducible projective variety over a number field \(K\). (In~\S\ref{birheightfacts}, we present results specific to the case where \(X\) is a rational surface equipped with a birational self-map.)  The following definitions and facts are presented in detail in~\cite{BG}. 

Let \(M_K\) be the set of normalized, nontrivial absolute values on \(K\), as defined in~\cite[\S1.4]{BG}. 
For $v\in M_K$, let $K_v$ be the completion of $K$ with respect to the norm $|\cdot|_v$.
For every \(a \in K \smallsetminus \{0\}\), the \emph{product formula} (see~\cite[Proposition 1.4.4]{BG}) states \[\sum_{v \in M_K} \log |a|_v =0,\]
where, furthermore, only finitely many terms are nonzero.
 
Given a base point free line bundle $L$ on $X$, a finite set $\Sigma$ of global sections spanning \(H^0(X,L)\), a nonzero global section \(\sigma \in H^0(X,L)\), and $v\in M_K$, 
we define the \textit{local height function}
\begin{equation*}
  h_{\sigma,\Sigma,v}(x) 
  := \log \max_{\tau\in\Sigma}\frac{|\tau(x)|_v}{|\sigma(x)|_v}
\end{equation*}
for $x \in X(K_v)$ with $\sigma(x)\ne0$. 
Note that $h_{\sigma,\Sigma,v}(x)\to+\infty$ as $x$ approaches $\Div(\sigma)$ in the $v$-adic
topology on $X(K_v)$.

We also define a \textit{global height function} by 
\[h_{\sigma,\Sigma}(x) := \sum_{v \in M_K} h_{\sigma,\Sigma,v}(x)\]
for $x\in X(K)$ with $\sigma(x)\ne0$, the sum again being finite.
The product formula implies that $h_{\sigma,\Sigma}$ does not depend on $\sigma$ in the sense 
that, if $\sigma,\sigma' \in H^0(X,L)$, then 
$h_{\sigma',\Sigma}(x)=h_{\sigma,\Sigma}(x)$ for $x\in X(K)$ with $\sigma(x),\sigma'(x)\ne0$.
Since $L$ is base point free, we can then define a function 
$h_{L,\Sigma}$ on $X(K)$ by setting 
\begin{equation*}
  h_{L,\Sigma}(x) := h_{\sigma,\Sigma}(x) 
\end{equation*}
where $\sigma\in H^0(X,L)$ is any section with $\sigma(x)\ne0$.

Using the same formulas above, $h_{L,\Sigma}$ extends to a function on
$X(\bK)$. This function does depend on the choice of $\Sigma$, 
but if $\Sigma,\Sigma'$ are finite spanning subsets of $H^0(X,L)$, then 
$h_{L,\Sigma}-h_{L,\Sigma'}$ is a bounded function on $X(\bK)$.
Hence we have a well-defined element $h_L$ in the vector space $W_{X,\bK}$
of functions on $X(\bK)$ modulo bounded functions. 
The assignment $L\to h_L$ extends uniquely to an $\R$- linear map from
$\Pic(X)_{\R}$ to $W_{X,\bK}$.

The element $h_L$ associated to $L\in\Pic(X)_{\R}$ is called the \textit{Weil height} of $L$. 
Slightly abusively, we think of it as a function on $X(\bK)$.
The assignment $L\to h_L$ is functorial in the sense that if 
$\phi\colon Y\to X$ is a morphism of irreducible varieties over \(K\), 
then $h_{\phi^*L}=h_L\circ\phi+O(1)$ on $Y(\bK)$.
We shall also use
\begin{proposition}[\cite{BG}, Proposition 2.3.9]\label{heightbounds}
  If \(D \in \DivGP(X)_\R\) is effective, then $h_{[D]}$ is bounded from below on 
  $(X\smallsetminus\Supp(D))(\bK)$.
\end{proposition}

\subsection{Distance functions}\label{S102}
Now suppose $X$ is smooth. 
For any $v\in M_K$, we equip $X(K_v)$ with a metric $\Dist_v$ given by the
pullback of the Fubini-Study metric on $\mathbb{P}^N(K_v)$, for
some fixed embedding $X\hookrightarrow\mathbb{P}^N_K$.
This metric is not canonical, but any two embeddings give rise to 
metrics on $X(K_v)$ that are Lipschitz equivalent. More generally,
if $x\in X(K_v)$ and $(z_1,\dots,z_n)$ are local $v$-adic analytic 
coordinates at $x$, then $(z_1,\dots,z_n)$ define a bi-Lipschitz
homeomorphism of a neighborhood of $x$ in $X(K_v)$ onto a bidisc 
in $K_v^n$.

\begin{lemma}\label{L102}
  Let $L$ be a base point free line bundle on $X$,  
  $\Sigma\subset H^0(X,L)$ a spanning set, $\sigma\in H^0(X,L)$ 
  a nonzero section and $y\in X(K)$ a point at which $\sigma$ vanishes.
  Then, for every $v\in M_K$ there exists $D_v\in\R$
  such that 
  \begin{equation}\label{e104}
    \log\frac{|\sigma(x)|_v}{\max_{\tau\in\Sigma}|\tau(x)|_v}
    \le\log\Dist_v(x,y)+D_v
  \end{equation}
  for all $x\in X(K_v)$, where 
  $D_v$ depends on $\sigma$, $\Sigma$, $v$ and the distance function 
  on $X(K_v)$, but not on $x$ or $y$.
\end{lemma}
\begin{proof}
  The statement is local on $X(K_v)$. Pick any $\tau\in\Sigma$ such that 
  $\tau(y)\ne0$. Then $f:=\sigma/\tau$ is a rational function on $X$
  that is regular and vanishes at $y$ and the left-hand side of~\eqref{e104} is
  bounded above by $\log|f(x)|_v$ near $y$.
  Pick local $v$-analytic coordinates $(z_1,\dots,z_n)$ at $y$. 
  By considering the Taylor expansion of $f$ we see that 
  $\log|f(x)|_v\le\log\max_i|z_i(x)|+O(1)=\log\Dist_v(x,y)$,
  which completes the proof.
\end{proof}

\section{Algebraically stable maps on $\mathbb{P}^2$}\label{S301}
For the convenience of the reader, we here present the proof of Theorem~A 
in the case of an algebraically stable birational map 
$f\colon\mathbb{P}_\mathbb{C}^2\dashrightarrow\mathbb{P}_\mathbb{C}^2$
of degree $\lambda\ge2$.
This case avoids many of the technical difficulties encountered in the general case.

By assumption,~(AS1) is satisfied, the dynamical degree of $f$ is
$\lambda$, and we have $\theta^+=\theta^-=\cO_{\mathbb{P}^2}(1)$, 
up to scaling. In particular, $\theta^+$ and $\theta^-$ are ample, so~(AS2) is automatically satisfied.

Assume that $f$ is defined over a number field $K$, with a fixed embedding 
$K\hookrightarrow\mathbb{C}$. 
We may assume that $\Ind_f$ and $\Ind_{f^{-1}}$ are defined over $K$.
Let $M_K$ be the set of normalized nontrivial absolute values of $K$. For each $v\in M_K$,
let $K_v$ denote the completion of $K$ with respect to $v$ and let 
$\Dist_v$ denote the Fubini-Study distance on $X(K_v)$.
We shall prove that,  for every $v\in M_K$ and every point $y\in X(K)$ whose 
forward orbit is disjoint from $\Ind_f$, we have 
  \begin{equation}\label{e302}
    \sum_{k \geq 0} \lambda^{-k} \log \Dist_v\big( f^k(y),\Ind_f \big) > -\infty.\tag{$\mathrm{BD}_\nu$}
  \end{equation}
Letting $v$ be the absolute value on $K$ induced by the embedding $K\hookrightarrow\mathbb{C}$,
and taking $y\in\Ind_{f^{-1}}$, this implies that the Bedford-Diller energy condition~(BD) holds. 

Let $\pi:K^3\to\mathbb{P}_K^2$ be the projection, and pick a homogeneous polynomial mapping
$F:K^3\to K^3$ such that $\pi\circ F=f\circ\pi$. Note that $F$ is unique up to scaling
and that $F(a)=0$ iff $\pi(a)\in\Ind_f$.

For any $v\in M_K$, let $\|\cdot\|_v$ be the norm on $K_v^3$ defined by 
$\|(a_0,a_1,a_2)\|_v:=\max_i|a_i|_v$.
The height on $\mathbb{P}^2(K)$ is now defined by 
\begin{equation*}
  h(\pi(a)):=\sum_{v\in M_K}\log\|(a_0,a_1,a_2)\|_v.
\end{equation*}
By the product formula, the right hand side does not change if we multiply the $a_j$ 
by a common nonzero number from $K$. Thus the height $h$ is well defined.
It is equal to the global height $h_{\sigma,\Sigma}$ in~\S\ref{BGheights} with $L=\cO(1)$,
$\Sigma=\{z_0,z_1,z_2\}\subset H^0(\mathbb{P}^2,\cO(1))$ and $\sigma=z_i$ for some $i$.
Note that $h\ge0$ since we can assume that $a_i=1$ for some $i$.

For $y\in(\mathbb{P}^2\setminus I_f)(K)$ and $v\in M_K$, define
\begin{equation*}
  \varphi_v(y):=\frac1{\lambda}\log\|F(a)\|_v-\log\|a\|_v,
\end{equation*}
for any $a\in K^3$ with $\pi(a)=y$.
There exists a constant $C_v\ge0$ only depending on $F$ such that 
\begin{equation*}
  \varphi_v(y)\le C_v
\end{equation*}
for all $y\in(\mathbb{P}^2\setminus I_f)(K)$. Further, there exists a 
finite set $S\subset M_K$ such that $C_v=0$ for $v\in M_K\setminus S$.

Note that there exists a constant $D>0$ such that 
\begin{equation}\label{e301}
  \varphi_v\le\log\Dist_v(\cdot,I_f)+D
\end{equation}
on $(\mathbb{P}^2\setminus I_f)(K)$, for all $v\in M_K$.

Now consider any point $y\in\mathbb{P}^2(K)$ such that $f^n(y)\not\in I_f$
for $n\ge1$. Write $y=\pi(a)$ with $a\in K^3$.
For every $k\ge0$ we have
\begin{equation*}
  \sum_{v\in M_K}\varphi_v(y_k)=\frac1{\lambda}h(y_{k+1})-h(y_k),
\end{equation*}
where $y_k:=f^k(y)$.
Hence 
\begin{equation*}
  \sum_{k=0}^{n-1}\sum_{v\in M_K}\lambda^{-k}\varphi_v(y_k)
  =\frac1{\lambda^n}h(y_n)-h(y)\ge-h(y),
\end{equation*}
for every $n\ge 0$. 
Now, for every $v\in M_K$ we have 
\begin{equation*}
  \sum_{k=0}^{n-1}\lambda^{-k}\varphi_v(y_k)
  \le\sum_{k=0}^{n-1}\lambda^{-k}C_v
  \le 2C_v.
\end{equation*}
This implies that, for every $v\in M_K$ and every $n\ge 0$:
\begin{equation*}
  \sum_{k=0}^{n-1}\lambda^{-k}\varphi_v(y_k)
  \ge-h(y)-2\sum_{w\ne v}C_w
  \ge-h(y)-2\sum_{w\in M_K}C_w,
\end{equation*}
where the last sum is finite.
It follows that $\sum_{k=0}^\infty\lambda^{-k}\varphi_v(y_k)>-\infty$.
This implies that condition~\eqref{e302} holds,
in view of~\eqref{e301}. 
\begin{remark}
  Vigny~\cite{Vig15} proved that if a birational selfmap of $\mathbb{P}^2$ satisfies~(BD), then it 
  has exponential decay of correlations.
\end{remark}
\begin{remark}
  It is known~\cite[Proposition~4.5]{BD-A} that if $f$ is any birational selfmap of 
  $\mathbb{P}^2$ of degree $>1$,
  then $f\circ A$ satisfies the Bedford-Diller condition for all 
  $A\in\mathrm{Aut}(\mathbb{P}^2)$  outside a pluripolar set.
  However, such a result is not useful for our study since the pluripolar set in question could 
  a priori contain all automorphisms of $\mathbb{P}^2$ defined over $\bar{\mathbb{Q}}$.
\end{remark}

\section{Dynamics of the base curves}\label{S202}
We take \(X\) and \(f\) to be defined over an algebraically closed field of arbitrary characteristic. Additionally, we take \(f\) to be strictly birational with \(\lambda > 1\) and to satisfy (AS1) and (AS2). The main result here is a quite precise description of the dynamics on the base curves for \(f\) and \(f^{-1}\).

\begin{theorem}\label{basecurvesmain}
  We can write \(\cC^+ = \cCper \sqcup \cCexc^+\) and \(\cC^- = \cCper \sqcup \cCexc^-\)
  and there exists $N\ge 1$ such that the behavior of every \(C \in \cC^+ \cup \cC^-\) 
  under \(f\) is described as follows.
  \begin{itemize}
  \item[(1)] 
    If \(C \in \cCper\), then, for every $n\in\Z$,
    \(C \cap \Ind_{f^n}=C\cap R_{f^n} = \emptyset\) 
    and \(f^{n*}C=C_n\) for some $C_n\in\cCper$.
  \item[(2)] 
    If $C\in \cCper$ and $D\in\cCexc^+\cup\cCexc^-$, then $C\cap D= \emptyset$.
  \item[(3)] 
    If \(C \in \cCexc^+ \cap \cCexc^-\), then \(C \cap \Ind_{f^n} = \emptyset\) 
    for every \(n \in \mathbb{Z}\) and \(f^{n*}C=0\) for $|n|\ge N$.
  \item[(4)] 
    If \(C \in \cCexc^+ \smallsetminus \cCexc^-\), then
    \begin{itemize}
    \item[(a)] \(C \cap \Ind_{f^n} = \emptyset\) for every \(n \in \N\);
    \item[(b)] \(f^n_*C=0\) for $n\ge N$;
    \item[(c)] \((C\cdot\theta^-) > 0\); and
    \item[(d)] \((f^n)^*[C]\) is big and nef for $n\ge N$;
    \end{itemize}
  \item[(5)] 
    If \(C \in \cCexc^- \smallsetminus \cCexc^+\), then (4) holds with 
    \(n\) replaced by \(-n\) 
    and \(\theta^-\) replaced by \(\theta^+\).
  \end{itemize}
\end{theorem}
Note that upon passing to an iterate we will have $N=1$ and $C_n=C$ for all $C\in\cCper$.

Most of the statements in Theorem~\ref{basecurvesmain} are relatively
straightforward, but as we have not been able to locate them in the
literature, we provide complete proofs. The most delicate part is~(4d), which 
we prove in~\S\ref{AGpreimages}. 

In~\S\ref{S205} we use Theorem~\ref{basecurvesmain} to study the
Zariski closure of forward orbits, and in~\S\ref{AGdecomp}, 
we use Theorem~\ref{basecurvesmain} to obtain a 
decomposition of \(\theta^+\) in terms of effective divisor classes 
that are all either big and nef or periodic under~\(f^*\).

\subsection{Types of base curves}\label{AGmorefacts}
We let \(\cCexc^{\pm} \subseteq \cC^{\pm}\) be the set of all base curves that are \(f^{\pm m}\)-exceptional for some \(m \in \N\), and we let \(\cCper^{\pm} = \cC^{\pm} \smallsetminus \cCexc^{\pm}\).
\begin{lemma}\label{basecurvetypes}
For \(C \in \cC^+\) we have 
\(C \cap \Ind_{f^n} = \emptyset\) for all \(n \in \N\). 
\end{lemma}
\begin{proof}
  For every \(n \in \N\),
  \[0 = \lambda^n(C \cdot \theta^+) = (C \cdot \lambda^n\theta^+) = (C \cdot (f^n)^*\theta^+) = (f^n_*C \cdot \theta^+).\]
  If there were some \(x \in C \cap \Ind_{f^n}\), then \(f^n_*C \geq f^n(x)\) and
  \[0 = (f^n_*C \cdot \theta^+) \geq (f^n(x) \cdot \theta^+),\]
  which would contradict (AS2). So $C\cap\Ind_{f^n}=\emptyset$.
\end{proof}
\begin{lemma}\label{perbasecurve}
  For any \(C \in \cCper^+\), we have \(f(C)\in\cCper^+\).
\end{lemma}
\begin{proof}
  By Lemma~\ref{basecurvetypes}, $C\cap\Ind_f=\emptyset$. Since $C\not\in\cCexc^+$,
  $f_*C=f(C)$ is an irreducible curve, and
  \[(f(C) \cdot \theta^+) = (f_*C \cdot \theta^+) = (C \cdot f^*\theta^+) = (C \cdot \lambda \theta^+) = 0; \]
  so \(f(C)\) is again an element of \(\cCper^+\). (Note that \(f(C) \notin \cCexc^+\) by definition.)
\end{proof}
Thus \(f\) must permute the finitely many elements of \(\cCper^+\).
Clearly, Lemma~\ref{basecurvetypes} and Lemma~\ref{perbasecurve} remain true 
if we replace \(\cC^+\) by \(\cC^-\), \(\cCper^+\) by \(\cCper^-\),
and \(f\) by \(f^{-1}\).

\begin{lemma}\label{periodiccurves}
  For any \(C \in \cCper^+\), \(C \cap \Ind_{f^{-n}} = \emptyset\) for all \(n \in \N\).
\end{lemma}
\begin{proof}
  For $n\ge 0$ there exists $C_n\in\cCper^+$ such that 
  $f^n(C)=f^n_*C=C_n$. Thus
  \begin{equation*}
    (C_{n+1}\cdot K_X) 
    = (f_*C_n\cdot K_X) 
    = (C_n \cdot f^*K_X) 
    = (C_n \cdot K_X) - (C_n \cdot R_f)
    \le (C_n \cdot K_X),
  \end{equation*}
  with equality iff $C_n$ does not intersect the support of $R_f$;
  indeed, $C_n$ and $R_f$ are effective and $C_n$ is not an 
  irreducible component of $R_f$.

  Since $\cCper^+$ is finite, there exists $N>0$ such that $C_N=C$.
  Thus $(C_n\cdot K_X)=(C_{n+1}\cdot K_X)$ for $0\le n<N$, 
  so $C_n$ does not intersect the support of $R_f$ for $0\le n<N$.

  If there were a point \(x \in C_{n+1}\cap \Ind_{f^{-1}}\) for some $n$, 
  then \(f^{-1}(x)\) would be contained in the support of \(R_f\) 
  and also contain a point in \(C_n\), a contradiction. 
  So \(f^{-1}\) is defined on all of \(C_{n+1}\) and
  \(f^{-1}(C_{n+1})=f^*C_{n+1}=C_n\).
  As a consequence, $C\cap\Ind_{f^{-n}}=\emptyset$ for all \(n \in \N\).
\end{proof}
\begin{corollary}\label{C201}
  We have $\cCper^+=\cCper^-=:\cCper$. 
  For any \(C \in \cCper\), \(C \cap \Ind_{f^n} = C\cap R_{f^n}=\emptyset\) 
  for all \(n \in \Z\). 
\end{corollary}
\begin{proof}
  If $C\in\cCper^+$, then there exists $n\ge1$ such that $f^{n*}C=f^n(C)=C$.
  Then Lemma~\ref{periodiccurves} shows that $C=f^{-n}(C)=f^n_*C$. So
  \begin{equation*}
    (C\cdot\theta^-)
    =(f^{n*}C\cdot\theta^-)
    =(C\cdot f^n_*\theta^-)
    =\lambda^n(C\cdot\theta^-);
  \end{equation*}
  hence $(C\cdot\theta^-)=0$ and $C\in\cCper^-$. 
  We have proved $\cCper^+\subset\cCper^-$ and the reverse inclusion
  follows by considering $f^{-1}$.

  Lemmas~\ref{basecurvetypes} and~\ref{periodiccurves} now give
  $C\cap\Ind_{f^n}=\emptyset$ for all $n\in\Z$.
  Finally, if $C\cap R_{f^n}\ne\emptyset$, then 
  $f^n(C)\cap\Ind_{f^{-n}}\ne\emptyset$, 
  which is impossible since $f^n(C)\in\cCper$.
\end{proof}
All the assertions in Theorem~\ref{basecurvesmain} now follow easily,
with the exception of~(4d) (and its analogue for $f^{-1}$), 
which will be proved in the next section: see Proposition~\ref{bigandnef}.

\subsection{Preimages of base curves}\label{AGpreimages}
For a prime divisor \(D\) on \(X\), the strict transform \(f^\#D\) cannot be an \(f\)-exceptional curve. On the other hand, \(f^\#D\) is an \(f^2\)-exceptional curve if and only if it is non-zero and \(D\) is \(f\)-exceptional. For every \(k \in \N\), it follows from (AS1) that \((f^k)^\# = (f^\#)^k\) and every \(f^{k-1}\)-exceptional curve is also \(f^k\)-exceptional. If \(D\) is \(f\)-exceptional, then \((f^k)^\#D\) and \((f^{k'})^\#D\) are distinct whenever \(k' > k \geq 0\) and \((f^k)^\#D \neq 0\).

\begin{lemma}\label{SNC}
For every \(x \in X \smallsetminus \Ind_f\), there are at most two distinct irreducible curves on \(X\) that contain \(x\) and are contained in the support of \(R_f\).
\end{lemma}

\begin{proof} Since \(g\) is a finite composition of point blow-ups, the support of \(R_g\) (i.e., the union of all of the \(g\)-exceptional curves) is a (not necessarily connected) simple normal crossing divisor on \(Y\). So, for any \(y \in Y\), there are at most two distinct irreducible curves on \(Y\) that contain \(y\) and are contained in the support of \(R_g\). Since the support of \(R_f\) is the image under \(\pi\) of the support of \(R_g\) and \(\pi\) is an isomorphism away from the preimage of \(\Ind_f\), the desired conclusion follows.
\end{proof}

\begin{lemma}\label{L101}
Let \(D\) be an \(f\)-exceptional prime divisor. Then there is \(\zeta_D \in \mathbb{Z}\) such that
\[((f^k)^\#D \cdot (f^k)^\#D) \geq \zeta_D\]
for every \(k \in \N_0\).
\end{lemma}

\begin{proof} If \((f^{N+1})^\#D=0\) for some \(N \in \N_0\), then
\[\zeta_D = \min \{(D \cdot D),\dots,((f^N)^\#D \cdot (f^N)^\#D),0\}\]
gives the desired inequality.

Otherwise, we write \(D_k := (f^k)^\#D\) for every \(k \in \N_0\); so \(\{D_k\}_{k \in \N_0}\) is a pairwise distinct collection of prime divisors on \(X\). It follows from Lemma~\ref{SNC} (applied to forward iterates of \(f\)) that each \(x \in \Ind_{f^{-1}}\) is contained in at most two distinct elements of \(\{D_k\}_{k \in \N_0}\). So, since \(\Ind_{f^{-1}}\) is finite, there is \(N \in \N\) such that \(D_k \cap \Ind_{f^{-1}} = \emptyset\) for all \(k \geq N\). It follows that
\[D_k = (f^{k-N})^*D_N,\]
and therefore $(D_k \cdot D_k) \geq (D_N \cdot D_N)$, for all \(k \geq N\). 
So
\[\zeta_D = \min \{(D \cdot D),\dots,(D_N \cdot D_N)\}\]
gives the desired inequality.
\end{proof}

Let \(\zeta\) be the minimum value of \(\zeta_D\) where \(D\) ranges over the finitely many irreducible curves contained in the support of \(R_f\). Then \((D' \cdot D') \geq \zeta\) whenever \(D'\) is an irreducible curve contained in the support of \((f^k)^*R_f\) for some \(k \in \N_0\). Note that \(\zeta < 0\) if \(\cCexc^+\) is non-empty.

\begin{proposition}\label{bigandnef}
For any \(C \in \cCexc^+\smallsetminus\cCexc^-\), the line bundle 
\(f^{n*}[C]\) is big and nef for $n\gg0$.
\end{proposition}
With this result in hand (and the analogue for $f^{-1}$), 
the proof of Theorem~\ref{basecurvesmain} is
complete.
\begin{proof} 
Our assumptions imply \(C \notin \cC^-\), that is, \((C \cdot \theta^-) >0\). 
We have
\[\lim_{n \rightarrow \infty} \lambda^{-n} (f^n)^*[C] = \frac{(C \cdot \theta^-)}{(\theta^+ \cdot \theta^-)}\theta^+.\]
Since \((\theta^+ \cdot \theta^+) > 0\), the continuity of the intersection product implies
\((f^{n*}C \cdot f^{n*}C)>0\) for \(n\gg0\).

It remains to be shown that $f^{n*}[C]$ is nef for $n\gg0$, that is, 
$(f^{n*}C\cdot D)\ge0$ for any irreducible curve $D$. We may assume that $D$
is contained in the support of $f^{n*}C$ and that $(D\cdot D)<0$. 
For such $D$ we have $\zeta\le(D\cdot D)\le -1$.
Furthermore, it is clear that $D\not\in\cCper$ and
$$(D \cdot f^{n*}C) = (f_*D \cdot (f^{n-1})^*C) = 0$$
for $n\ge1$ and $D \in \cCexc^+$. Thus we may assume 
$D\not\in\cC^+$, so that $(D\cdot\theta^+)>0$.
Let $\mathcal{D}$ be the set of irreducible curves $D$ on \(X\) with \(\zeta \leq (D \cdot D) \leq -1\) 
and \((D \cdot \theta^+)>0\).

To simplify the notation, we normalize $\theta^+$ by $(\theta^+\cdot\theta^+)=1$.
Given $s\le-1$ and $t\le0$, consider subsets $V_s$ and $W_t$ of $\Pic(X)_{\mathbb{R}}$ defined by 
\begin{align*}
  V_s &:= \{w\in\Pic(X)_{\R}\mid s\leq (w \cdot w) \leq -1 , 0 \leq (w \cdot \theta^+) \leq 1\}\\
  W_t &:= \{w\in\Pic(X)_{\R}\mid t\leq (w \cdot w) \leq 0,\ (w \cdot \theta^+) =1\}.
\end{align*}
We claim that $V_s$ and $W_t$ are compact.
To see this, let \(\{v_1,\dots,v_M\}\) be an orthonormal basis for the negative definite 
space $\theta^{+\perp}\subseteq \Pic(X)_\R$. Using the basis \(\{\theta^+,v_1,\dots,v_M\}\) for $\Pic(X)_{\R}$ we have
\begin{align*}
  V_s &\simeq \{(a_0,\dots,a_M)\in\R^{M+1}
\mid 0 \leq a_0 \leq 1, \, 1+a_0^2 \leq a_1^2 +\dots+ a_M^2 \leq a_0^2-s\},\\
  W_t &\simeq \{(a_0\dots,a_M)\in\R^{M+1}
\mid a_0=1, 1-t\leq a_1^2 + \dots + a_M^2 \le1\},
\end{align*}
from which it clear that $V_s$ and $W_t$ are compact.

Now \(\Pic(X)\) is discrete in \(\Pic(X)_\R\), so it intersects $V_\zeta$ in a finite set. 
It follows that there exists \(\epsilon > 0\) such that \((D\cdot \theta^+)\ge\epsilon\) for all $D\in\mathcal{D}$. 

Given $D\in\mathcal{D}$, set $w_D:=(D\cdot\theta^+)^{-1}[D]$, so that 
$w_D\in W:=W_{\zeta/\epsilon^2}$.
By the continuity of the intersection form, and the compactness of $W$,
there exists a neighborhood $U$ of $\theta^+$ in $\Pic(X)_{\R}$ such that 
$(\beta\cdot w)\ge 0$ for all $w\in W$ and all $\beta\in U$.
For $n\gg0$ and $D\in\mathcal{D}$, suitable positive multiples of $f^{n*}[C]$ and $[D]$
lie in $U$ and $W$, respectively; hence $(f^{n*}C\cdot D)\ge0$, which
completes the proof.
\end{proof}

\subsection{Structure of forward orbits}\label{S205}
\begin{lemma}\label{pointorbits}
  Consider a point $x\in X$ with a well-defined forward orbit, and let $Z$ be the Zariski closure
  of this orbit. Then exactly one of the following assertions holds:
  \begin{itemize}
  \item[(1)] 
    $x$ is preperiodic and hence $Z$ is a finite set;
  \item[(2)]
    $Z$ is a union of curves in $\cCper$;
  \item[(3)] 
    $Z$ is an infinite set intersecting $\Ind_f$, and
    for $n\gg0$, $f^n(x)$ does not belong to any of the curves in \(\cC^+\).
  \end{itemize}
\end{lemma}
\begin{proof} 
  Let $I\subset\N$ be the times $n$ such that $f^n(x)$ belongs to a curve 
  in $\cC^+$. 
  First suppose $I$ is infinite.
  If $f^n(x)$ lies on a curve $C\in\cCper$ for some $n\in I$, then either the orbit is 
  preperiodic or we are in case~(2).
  Now assume that the orbit is disjoint from the curves in $\cCper$.
  Then there exist $m,n\in I$ with $0\le m<n$, and a curve $C\in\cCexc^+$
  such that $f^m(x),f^n(x)\in C$. This implies that the orbit is preperiodic
  (satisfying \(f^{m+1}(x)=f^{n+1}(x)\)).

  Finally suppose $I$ is finite. If $Z$ is a finite set, then the orbit is preperiodic, so
  suppose $Z$ has dimension at least one. We must show that $Z$ intersects $\Ind_f$, 
  so that we are in case~(3). This is clear if $Z=X$. Now suppose $Z$ is a curve disjoint from 
  $\Ind_f$ and let $C$ be one of its irreducible component curves. Since $f_*Z=f(Z)=Z$, $C$ cannot be $f$-exceptional and $f(C)$ must be another curve in $Z$.
  It follows that $f^n_*C=C$ for some $n\ge 1$,
  and then 
  \begin{equation*}
    \lambda^n(\theta^+\cdot C)
    =(f^{n*}\theta\cdot C)
    =(\theta^+\cdot f^n_*C)
    =(\theta^+\cdot C);
  \end{equation*}
  hence $(\theta^+\cdot C)=0$ since $\lambda>1$.
  Thus $C\in\cCper$, which contradicts the assumption
  that $I$ is finite.
\end{proof}

\subsection{A new decomposition for the invariant class}\label{AGdecomp}
Next we pull-back the decomposition of $\theta^+$ in Proposition~\ref{Kawdecomp}
to obtain a new decomposition better suitable for height considerations.
\begin{proposition}\label{newdecomp}
  We have a decomposition
  \begin{equation*}
    \theta^+ = \sum_{i=1}^p r_i B_i + L + [\Dper]
    \end{equation*}
    with the following properties:
    \begin{itemize}
    \item[(1)] for each $i$, $r_i>0$; $B_i$ and $f^*B_i$ are base point free;
      and $(B_i\cdot f(x))>0$ for every $x\in\Ind_f$;
    \item[(2)] $L\in\Pic(X)_{\R}$ is nef;
    \item[(3)] $\Dper$ is an effective $\R$-divisor supported on the curves in $\cCper$.
\end{itemize}
\end{proposition}

\begin{proof} 
  After scaling, Proposition~\ref{Kawdecomp} yields
  $\lambda^{-1}\theta^+=A+[D]$, where $A\in\Pic(X)_{\R}$ is ample
  and $D$ is an effective $\R$-divisor supported on the curves in $\cC^+$.
  Thus $\theta^+=\lambda^{-1}f^*\theta^+=f^*A+f^*[D]$.
  
  Now use Theorem~\ref{basecurvesmain}. 
  The decomposition $\cC^+=\cCper\cup\cCexc^+$
  gives rise to a decomposition $D=\Dper+\Dexc$.
  Pick $n\ge 1$ such that $f^{-n}(C)=C$ for $C\in\cCper$,
  $f^{n*}C=0$ for $C\in\cCexc^+\cap\cCexc^-$ and 
  $f^{n*}C$ is (big and) nef for $C\in\cCexc^+\smallsetminus\cCexc^-$. 
  Then $f^{n*}[\Dper]=[\Dper]$ and $L:=f^{n*}[\Dexc]$ is nef.

  Finally we consider $f^{n*}A$.
  Write $A=\sum_{i=1}^pr_iA_i$, where $r_i>0$ and $A_i\in\Pic(X)$ is ample.
  This obviously implies $f^{n*}A=\sum_{i=1}^pr_iB_i$, where $B_i=f^{n*}A_i$.
  Clearly $B_i$ is then nef, but may not have the other required properties.
  
  However, we may scale the situation, by which we mean replacing  
  $A_i$ by $mA_i$ and $r_i$ by $r_i/m$,  
  for  a large and divisible integer $m$. 
  In particular, we may assume $A_i$ is very ample. 
  Then the base locus of \(B_i=f^{n*}A_i\) (resp., $f^*B_i=f^{(n+1)*}A_i$) is a subset of 
  \(\Ind_{f^n}\) (resp., $\Ind_{f^{n+1}}$) and hence finite.
  A result by Zariski (see~\cite[Remark~2.1.32]{Laz}) then implies that $mB_i$ and $mf^*B_i$
  are base point free for $m$ sufficiently divisible.
  After scaling, $B_i$ and $f^*B_i$ are base point free for all $i$.
  
  Finally, for any $x\in\Ind_f$, $f(x)$ is a nonzero (reduced) effective divisor
  on $X$. The effective divisor $f^n_*f(x)$ is also nonzero, since we have 
  \begin{equation*}
    (\theta^+\cdot f^n_*f(x))
    =(f^{n*}\theta^+\cdot f(x))
    =\lambda^n(\theta^+\cdot f(x))>0
  \end{equation*}
  by~(AS2). 
  Since $A_i$ is ample, this implies 
  $(B_i\cdot f(x))=(A_i\cdot f^n_*f(x))>0$ for all $i$, which completes the proof.
\end{proof}

\subsection{Examples}\label{examples}

We highlight some existing examples in the literature which exhibit the different types of base curves described in Theorem~\ref{basecurvesmain}.

Consider the involutions of $\mathbb{P}^2$
$$\sigma: [x:y:z] \mapsto [xyz+(-y+z)x^2:xyz+(-x+z)y^2:xyz],$$
$$\tau: [x:y:z] \mapsto [x:bx+(a+1)z-y:z].$$
Bedford and Diller~\cite{BD-B} described the dynamics of $f=\sigma \circ \tau$ for generic $a,b$ such that $f$ is strictly birational with $\lambda(f)>1$. After blow-ups at the points $[1:0:0],[1:b:0],[0:1:0]$ plus another point infinitely near to $[0:1:0]$, $f$ satisfies both (AS1) and (AS2). For this model of $f$, $\cCexc^+ \cup \cCexc^-$ is empty while $\cCper$ contains exactly two curves: the strict transform of $\{z=0\}$ and the exceptional curve over $[0:1:0]$.

Consider also the birational selfmap on $\mathbb{P}^1 \times \mathbb{P}^1$ given in affine coordinates by
$$f: (z,w) \mapsto \left( w+1-\epsilon,z \frac{w-\epsilon}{w+1} \right).$$
Diller and Favre~\cite{DF} described the dynamics of $f$, which is strictly birational with $\lambda(f)>1$ for all but finitely many $\epsilon$. The $f$-exceptional curves are $\{w=\epsilon\}$ and $\{w=-1\}$; the $f^{-1}$-exceptional curves are $\{z=-\epsilon\}$ and $\{z=1\}$. When $\epsilon = 1/k$ with $k \geq 4$ an integer, the forward orbit of $\{w=\epsilon\}$ meets $\Ind_{f^{-1}}$, and blowing up the points in this orbit yields a map satisfying (AS1) and (AS2) with nontrivial $\cC^{\pm}$. In this case,  the strict transform of $\{w=-1\}$ is the unique curve in $\cCexc^+ \smallsetminus \cCexc^-$, the strict transform of $\{z=1\}$ is the unique curve in $\cCexc^- \smallsetminus \cCexc^+$, the strict transform of the line through (1,0) and (0,-1) is the unique curve in $\cCper$, and $\cCexc^+ \cap \cCexc^-$ is empty. The sets $\cC^{\pm}$ are similar if $\epsilon = k/(k+2)$ with $k \geq 3$ an integer (in which case (AS1) and (AS2) are achieved via blow-ups along the forward orbit of $\{w=-1\}$).

\section{A general energy condition}\label{S203}
In this section we formulate and prove a general energy condition for birational surface maps 
defined over number fields. As shown in~\S\ref{S204}, this will imply Theorem~\ref{MainThm}.
We shall also prove Theorem~\ref{HeightThm}, establishing the existence of
a canonical height function.

Let $X$ be a rational surface and $f:X\dashrightarrow X$ be birational map 
of dynamical degree $\lambda>1$, defined over a number field $K$. 
Fix an algebraic closure $\bar K$ of $K$. 
After passing to a finite extension of $K$ inside $\bar K$, we may and will
assume that $\Ind_f$ and $\Ind_{f^{-1}}$ are defined over $K$.
\begin{theorem}\label{T101}
  Assume that $f$ is strictly birational and satisfies~(AS1) and~(AS2). 
  Then, for every $v\in M_K$ and every point $y\in X(K)$ whose forward orbit is disjoint 
  from $\Ind_f$, we have 
  \begin{equation}\label{e103}
    \sum_{k \geq 0} \lambda^{-k} \log \Dist_v\big( f^k(y),\Ind_f \big) > -\infty.\tag{$\mathrm{BD}_\nu$}
  \end{equation}
\end{theorem}
Here $\Dist_v$ is a distance on $X(K_v)$ as in~\S\ref{S102}. By making finite field extensions,
the estimate~($\mathrm{BD}_v$) holds for all points $y\in X(\bar K)$ with forward orbit disjoint from $\Ind_f$.

For the proof of Theorem~\ref{T101} we may and will replace $K$ by a finite extension inside 
$\bar K$ such that ``everything in sight'' is defined over $K$, notably all curves in $\cC^\pm$, 
the morphisms $\pi:Y\to X$ and $g:Y\to X$, etc.

\subsection{Growth of certain Weil heights}\label{birheightfacts}
\begin{lemma}\label{periodichtbd}
  For \(C \in \cCper\) the function $h_{[C]}\circ f-h_{f^*[C]}$ is 
  bounded on $(X\smallsetminus\Ind_f)(\bK)$.
\end{lemma}
\begin{proof} 
  By Theorem~\ref{basecurvesmain}, \(f^*C=C'\) for some $C'\in\cCper$. Further, 
  \(C' \cap \Ind_f = C\cap\Ind_{f^{-1}}=\emptyset\); hence
  \(g^*C=\pi^*C'\).
  For $x\in(X\setminus\Ind_f)(\bK)$ there exists a unique $y\in Y(\bK)$ with $\pi(y)=x$
  and $g(y)=f(x)$. Since Weil heights are functorial with respect to morphisms, we get
  \begin{equation*}
    h_{[C]}(f(x))
    =h_{[C]}(g(y))
    =h_{g^*[C]}(y)+O(1)\\
    =h_{\pi^*[C']}(y)+O(1)
    =h_{[C']}(\pi(y))+O(1)
    =h_{[C']}(x)+O(1),
  \end{equation*}
  which completes the proof.
\end{proof}
The following result is a special case of Proposition 21 and Remark 23 in~\cite{KS}.
\begin{lemma}\label{nefhtbd}
  If  $L\in\Pic(X)_{\R}$ is nef,
  then $h_L\circ f-h_{f^*L}$ is bounded above on $(X\setminus\Ind_f)(\bK)$.
\end{lemma}
\begin{proof} 
  Since $L$ is nef, we have \(g^*L \leq \pi^*f^*L\) by~\eqref{e102}, and
  \(\pi^*f^*L-g^*L\) is the class of an effective $\pi$-exceptional divisor.
  Proposition~\ref{heightbounds} shows that 
  $h_{\pi^*f^*L}-h_{g^*[L]}$ is bounded above on 
   $(Y\smallsetminus\Supp(R_\pi))(\bK)$.
   The result now follows from the functoriality of Weil heights with respect to morphisms,
   as in the proof of Lemma~\ref{periodichtbd}.
\end{proof}

\subsection{Growth of certain local heights}\label{dynlocheights}
Consider a line bundle $B\in\Pic(X)$ satisfying condition~(1) in Proposition~\ref{newdecomp},
that is, $B$ and $f^*B$ are base point free and $(B\cdot f(x))>0$ for all $x\in\Ind_f$.

Let $\Sigma\subset H^0(X,B)$ and $\Sigma'\subset H^0(X,f^*B)$ be 
finite spanning sets of nonzero global sections such that: $\tau'\in\Sigma'$ for every
$\tau\in\Sigma$, for the operation $\tau\mapsto\tau'$ described in Proposition~\ref{P101}; and for each $\tau \in \Sigma$, $\Div(\tau)$ is not supported on any $f^{-1}$-exceptional curve.
Fix a reference section \(\sigma \in H^0(X,B)\), with associated section $\sigma'\in H^0(X,f^*B)$. 

\begin{proposition}\label{lochtbd}
  For each $v\in M_K$ there exist
  $C_v=C_{\sigma,\Sigma,\Sigma',v}\in\R$ and
  $D_v=D_{\sigma,\Sigma,\Sigma',v}\in\R$
  such that 
  \begin{equation*}
    h_{\sigma,\Sigma,v}(f(x))-h_{\sigma',\Sigma',v}(x)
    \le C_v + \min\{0,\log\Dist_v(x,\Ind_f)+D_v\}
  \end{equation*}
  for every $x \in X(K_v)$ such that $x\not\in\Ind_f$ and $\sigma(f(x))\ne0$. 
  Moreover, there is a finite set \(S=S_{\sigma,\Sigma} \subseteq M_K\) such that 
  $C_v=0$ for all $v\not\in S_\sigma$.
\end{proposition}
\begin{proof}
  For any $\tau\in\Sigma\subset H^0(X,B)$, Lemma~\ref{L103} shows that 
  the quantity
  \begin{equation*}
    a_{\sigma,\tau}
    :=\frac{\tau(f(x))/\sigma(f(x))}{\tau'(x)/\sigma'(x)}
  \end{equation*}
  is well defined and independent of $x$, as long as $x\not\in\Ind_f$,
  $\sigma(f(x))\ne 0$, and $\tau(f(x))\ne0$.
  Set
  \begin{equation*}
    C_v:=\max_{\tau\in\Sigma}\log|a_{\sigma,\tau}|_v
  \end{equation*}
  for $v\in M_K$. Then $C_v=0$ for all but finitely many $v$.
  Further, if $x\not\in\Ind_f$ and $\sigma(f(x))\ne 0$, then
  \begin{align*}
    h_{\sigma,\Sigma,v}(f(x))-h_{\sigma',\Sigma',v}(x)
    &=\log\max_{\tau\in\Sigma}\frac{|\tau(f(x))|_v}{|\sigma(f(x))|_v}
    -\log\max_{\tau'\in\Sigma'}\frac{|\tau'(x)|_v}{|\sigma'(x)|_v}\\
    &=\log\max_{\tau\in\Sigma}|a_{\sigma,\tau}|_v\cdot\frac{|\tau'(x)|_v}{|\sigma'(x)|_v}
    -\log\max_{\tau'\in\Sigma'}\frac{|\tau'(x)|_v}{|\sigma'(x)|_v}\\
    &\le C_v+\log\max_{\tau\in\Sigma}\frac{|\tau'(x)|_v}{|\sigma'(x)|_v}
   -\log\max_{\tau'\in\Sigma'}\frac{|\tau'(x)|_v}{|\sigma'(x)|_v}\\
    &=C_v+\log\frac{\max_{\tau\in\Sigma}|\tau'(x)|_v}{\max_{\tau'\in\Sigma'}|\tau'(x)|_v}.
  \end{align*}
  The right hand side is clearly bounded above by $C_v$.
  Further, it follows from (2) in Proposition~\ref{P101} that, for any $y\in\Ind_f$, 
  the sections $\tau'$, for $\tau\in\Sigma$, all vanish at $y$,
  whereas at least one of the sections in $\Sigma'$ does not vanish there.
  Then Lemma~\ref{L102} shows that the 
  right hand side is bounded above by 
  $C_v+\log\Dist_v(x,\Ind_f)+D_v$, completing the proof.
\end{proof}

\subsection{Proof of Theorem~\ref{T101}}
Pick any $y\in X(K)$ with well-defined forward orbit.
If the Zariski closure of the orbit is disjoint from $\Ind_f$, then~\eqref{e103} is trivial. 
By Lemma~\ref{pointorbits} we therefore assume that,
for $n\gg0$, $f^n(y)$ does not lie on any of the curves in $\cC^+$.
Now~\eqref{e103} holds for $f^n(y)$ iff it holds for $y$, so
we may assume the forward orbit of $y$ does not intersect any of the curves in $\cC^+$.

First we use the decomposition $\theta^+=A+[D]$  
from Proposition~\ref{Kawdecomp}. 
Since $A$ is ample and the forward orbit of $y$ is disjoint from the support of $D$,
Proposition~\ref{heightbounds} implies that 
$h_{\theta^+}$ is bounded below along the forward orbit of $y$. Thus 
\begin{equation*}
  h^+(n) := \lambda^{-n}h_{\theta^+}(f^n(y)) - h_{\theta^+}(y)
\end{equation*}
is uniformly (in $n$) bounded below.

Next we use the decomposition  $\theta^+=\sum_{i=1}^p r_i B_i + L + [\Dper]$ from 
Proposition~\ref{newdecomp}. Write $B:=\sum_{i=1}^pr_iB_i$.
For \(k \in \N\), let
\begin{align*}
  \alpha_k &:= h_B(f^k(y)) - h_{f^*B}(f^{k-1}(y))\\
  \beta_k &:= h_L(f^k(y)) - h_{f^*L}(f^{k-1}(y))\\
  \gamma_k &:= h_{[\Dper]}(f^k(y)) - h_{f^*[\Dper]}(f^{k-1}(y))\\
  \delta_k &:= h_{f^*\theta^+}(f^{k-1}(y)) - \lambda h_{\theta^+}(f^{k-1}(y)).
\end{align*}
So
\begin{equation*}
  h^+(n) 
  = \sum_{k=1}^n \lambda^{-k}\alpha_k 
  + \sum_{k=1}^n \lambda^{-k}\beta_k 
  + \sum_{k=1}^n \lambda^{-k}\gamma_k 
  + \sum_{k=1}^n \lambda^{-k}\delta_k
\end{equation*}
for every \(n \in \N\). 

First, since \(f^*\theta^+=\lambda \theta^+\), there is a uniform (in \(k\)) upper bound on \(\delta_k\).
Second, since $\Dper$ is supported on curves in $\cCper$, 
Lemma~\ref{periodichtbd} gives a uniform (in \(k\)) upper bound on \(\gamma_k\). 
Third, Lemma~\ref{nefhtbd} gives a uniform (in \(k\)) upper bound on \(\beta_k\). 
Thus there must be a uniform (in \(n\)) lower bound on
\[\sum_{k=1}^n \lambda^{-k}\alpha_k.\]

We use notation as in~\S\ref{dynlocheights}.
For each $i$, pick a basis $\Sigma_i$ of \(H^0(X,B_i)\) and a 
spanning set $\Sigma'_i$ of \(H^0(X,f^*B_i)\) with the properties stated in~\S\ref{dynlocheights}. 
For each \(k \in \N\), pick $\sigma_{i,k}\in\Sigma_i$
such that \(\sigma_{i,k}(f^k(y)) \neq 0\). 
Then $\sigma'_{i,k}(f^{k-1}(y))\ne0$ and 
\begin{equation*}
  \alpha_k 
  = \sum_{i=1}^pr_i
  \sum_{v \in M_K}
  \left(
    h_{\sigma_{i,k},\Sigma_i,v}(f^k(y))
    -h_{\sigma'_{i,k},\Sigma'_i,v}(f^{k-1}(y))
  \right)
\end{equation*}
for every \(k \in \N\).
Proposition~\ref{lochtbd} gives a uniform (in \(k\)) upper bound on
\begin{equation*}
  \alpha'_k 
  = \sum_{i=1}^pr_i
  \sum_{w\ne v}
  \left(
    h_{\sigma_{i,k},\Sigma_i,w}(f^k(y))
    -h_{\sigma'_{i,k},\Sigma'_i,w}(f^{k-1}(y))
  \right);
\end{equation*}
indeed, the inner sum is non-positive for all but finitely many $w$.
Thus there must be a uniform (in \(n\)) lower bound on
\begin{equation*}
  \sum_{k=1}^n \lambda^{-k}(\alpha_k-\alpha_k') 
  = \sum_{k=1}^n \lambda^{-k} 
  \sum_{i=1}^pr_i
  \left(
    h_{\sigma_{i,k},\Sigma_i,v}(f^k(y))
    -h_{\sigma'_{i,k},\Sigma'_i,v}(f^{k-1}(y))
  \right)
\end{equation*}
It then follows from Proposition~\ref{lochtbd} that
\begin{equation*}
  \sum_{k \geq 0} \lambda^{-k} \log \Dist_v \big( f^k(y),\Ind_f \big) > -\infty,
\end{equation*}
completing the proof.

\subsection{Proof of Theorem~\ref{HeightThm}}\label{PfHeightThm}
Let $f_{\bK}: X_{\bK} \dashrightarrow X_{\bK}$ be the base change to $\bK$. Since we are only concerned with the birational conjugacy class,
we can assume \(f_{\bK}\) satisfies (AS1) and (AS2). 
Since $\theta^+$ is nef, Lemma~\ref{nefhtbd} gives
\begin{equation*}
  h_{\theta^+}\circ f \leq \lambda h_{\theta^+} + O(1)
\end{equation*}
on $(X\smallsetminus\Ind_f)(\bK)$. 
It then follows from a standard dynamical argument (e.g., the proof of
Proposition 1.2 in~\cite{CS}) that 
\(\hh^+(x):=\lim_{n\to\infty}\lambda^{-n}h_{\theta^+}(f^n(x))\) is
well-defined and contained in $[-\infty,\infty)$ for all \(x \in
X(\bK)\) with well-defined forward orbit. 

To prove that $\hh^+$ is nonnegative we use the decomposition $\theta^+=A+[D]$ from
Propositions~\ref{Kawdecomp}. Here $A$ is ample and $D$ is an effective
$\R$-divisor supported on the curves in $\cC^+$, so $h_{\theta^+}$ is
bounded below on points in $X(\bK)$ that do not lie on the curves in $\cC^+$.

Pick $x\in X(\bK)$ with well-defined forward orbit. If, for $n\gg0$,
$f^n(x)$ does not lie on any of the curves in $\cC^+$, then it is
clear that $\hh^+(x)\ge 0$. 
If $x$ is preperiodic, then \(\hh^+(x)=0\).  
By Lemma~\ref{pointorbits} the only remaining case is when
the Zariski closure $Z$ of the forward orbit of $x$ is a cycle of curves in \(\cCper\). But $f$ is an isomorphism of $Z$, so for any 
$\alpha\in\Pic(Z)_\R$ and any $x\in Z(\bK)$, $|h_\alpha(f^n(x))|$ grows at most linearly in $n$.  
The functoriality of Weil heights now gives 
\begin{equation*}
  \hh^+(x)
  =\lim_{n\to\infty}\lambda^{-n}h_{\theta^+}(f^n(x)) 
  =\lim_{n\to\infty}\lambda^{-n}h_{\theta^+|_{Z}}(f^n(x))
  =0.
\end{equation*}
This shows that $\hh^+$ is well-defined and
nonnegative. We obviously have 
$\hh^+\circ f=\lambda\hh^+$.
Lemma~\ref{pointorbits} and the preceding argument further show that $\hh^+(x)=0$
for every point $x\in X(\bK)$ such that the Zariski closure of the
forward orbit of $x$ does not intersect $\Ind_f$.

Now suppose $\hh^+\equiv 0$ on $X(\bK)$. We will show that 
$\lambda^{-n}h_A\circ f^n\to 0$ on $X(\bK)$ for every ample $A\in\Pic(X)_\R$.
If suffices to do this for one choice of $A$. Indeed, if $A$ and $B$
are ample, then there exists a constant $R>0$ such that $h_B\le
Rh_A+O(1)$. We may therefore pick $A$ as in the decomposition 
$\theta^+=A+[D]$ as above.
Pick any $x\in X(\bK)$ with well-defined forward orbit. 
Since $A$ is ample, $h_A$ is bounded below along the orbit of $x$.

If the orbit of $x$ has at most finite intersection with the support
of $D$, then $h_{[D]}$ is also bounded below along the orbit. 
Since $\lim_{n\to\infty}\lambda^{-n}h_{\theta^+}(f^n(x))=0$,
this implies $\lim_{n\to\infty}\lambda^{-n}h_A(f^n(x))=0$.
The latter equality trivially holds also if $x$ is preperiodic.

Finally, suppose the Zariski closure $Z$ of the orbit of $x$ is a cycle of curves in $\cCper$. Again since the $f$ is an
isomorphism of $Z$, the argument above shows that 
$\lim_{n\to\infty}\lambda^{-n}h_A(f^n(x))=0$.

A conjecture of Silverman (see~\cite[Conjecture~3, p.650]{Sil14}) would
imply that $\hh^+(x)>0$ for all points $x\in X(\bar K)$ with Zariski
dense orbit. The existence of such points was proved by 
Xie~\cite[Theorem~1.4]{XiePerpts}, using a result of Amerik~\cite{Amerik}.

\section{Complex dynamics}\label{S204}
\subsection{Proof of Theorem~\ref{MainThm}}\label{PfMainThm}
Our situation is as follows. 
We have a number field $K$ together with a fixed embedding \(K \hookrightarrow \C\). 
We also have a smooth projective surface $X$ and a birational selfmap $f:X \dashrightarrow X$, both 
defined over $K$. Let $f_\C:X_\C \dashrightarrow X_\C$ be the base change and assume that $f_\C$ has
dynamical degree $\lambda>1$.  We then want to show that $f_\C$ is birationally conjugate to 
a birational selfmap of a smooth complex projective surface satisfying the Bedford-Diller condition.

If $f_\C$ is birationally conjugate to a complex surface automorphism, then the Bedford-Diller energy condition is trivially satisfied (since automorphisms have no indeterminacy points). We may therefore assume $f_\C$ is strictly birational, in the sense that no such conjugation exists. 
Table 1 in~\cite{DF} shows that \(X_\C\), and hence $X$, must be a rational surface in this case. 

Extend the given embedding $K\subset\C$ to an embedding $\bK\subset\C$ 
of an algebraic closure of $K$. 
By~\cite[Theorem~0.1]{DF} we can make a finite number of blowups on $X_{\bK}$ 
such that the lift of $f_{\bK}$ satisfies~(AS1). Replacing $K$ by a finite extension inside $\bar K$, we may assume that the blowups are defined over $K$. We are therefore reduced to the case when $f_{\bK}$ satisfies~(AS1). 
Similarly, by~\cite[Proposition~4.1]{BD-A}, we can blow down curves on $X_{\bK}$ so that the induced selfmap satisfies both~(AS1) and~(AS2). After a finite field extension, we are then reduced to the case when $f_{\bK}$ itself satisfies both~(AS1) and~(AS2).

Let $\theta^\pm\in\Pic(X_{\bK})_\R$ be the invariant classes. We may assume $(\theta^+\cdot\theta^+)>0$, or else $f_{\bK}$, and hence $f_\C$,  is conjugate to an automorphism. We are now in position to invoke Theorem~\ref{T101} and conclude the proof.

\subsection{Proof of Corollary~\ref{MainCorCurr}}\label{PfMainCorCurr}

There is a birational morphism \(\rho:X \rightarrow X'\) such that
\(f' := \rho \circ f \circ \rho^{-1}\) satisfies (AS1) and (AS2). By
Theorem~\ref{MainThm}, \(f'\) has finite dynamical energy, as defined
in~\cite{DDG2}. By the main Theorem in~\cite{BD-A} and Proposition 2.2 and Theorem 2 in~\cite{DDG2}, the currents $T^{\pm}(f')$ are laminar and do not charge any pluripolar set, and the $f'$-invariant measure $\mu(f'):=T^+(f') \wedge T^-(f')$ is a geometric intersection. Proposition 2.8 in~\cite{BD-A} shows that \(\rho_*T^{\pm} = T^{\pm}(f')\), while Theorem 1 in~\cite{Fav-B} shows that neither \(T^+\) nor \(T^-\) charges any curve in the support of \(R_\rho\). It follows that the desired conclusions hold for \(T^{\pm}\) and \(\mu\).

\subsection{Proof of Corollary~\ref{MainCorMeas}}\label{PfMainCorMeas}

There is a birational morphism \(\pi:Y \rightarrow X\) such that \(g := \pi^{-1} \circ f \circ \pi\) satisfies (AS1), and there is a birational morphism \(\rho:Y \rightarrow Y'\) such that \(g' := \rho \circ g \circ \rho^{-1}\) satisfies (AS1) and (AS2). Take \(\theta^{\pm}(g)\) to be scaled so that \((\theta^+(g) \cdot \theta^-(g))=1\); taking \(\theta^{\pm}(g') = \rho_*\theta^{\pm}(g)\), it follows that \((\theta^+(g') \cdot \theta^-(g'))=1\) also. Set \(\nu' := T^+(g') \wedge T^-(g')\), let \(\nu\) be the strict transform of \(\nu'\) under \(\rho\) (so that \(\rho_*\nu = \nu'\)), and set \(\mu := \pi_*\nu\).

By Corollary~\ref{MainCorCurr}, \(\nu\) does not charge any pluripolar set, so both \(\nu\) and \(\mu\) are probability measures satisfying (1). The main Theorem in~\cite{BD-A} shows that \(\nu'\) satisfies (2); thus \(\nu\) and \(\mu\) also satisfy (2).

By Th\'eor\`eme 1 in~\cite{DS}, \(h_{\rm{top}}(f) \leq \lambda\). As noted in \S 4.2 in~\cite{Duj06}, we also have \(h_\mu(f) \leq h_{\rm{top}}(f)\). By Theorems 1 and 2 in~\cite{Duj06}, \(h_{\nu'}(g')=\log(\lambda)\), and thus (3) holds for \(\nu'\), \(\nu\) and \(\mu\). By Theorem 2 (see also Theorem 5.4) in~\cite{Duj06}, (4) holds for \(\nu'\), and hence also for \(\nu\) and \(\mu\), since these measures do not charge any curve.

\end{document}